\documentclass[reqno]{amsart}

\usepackage{amsfonts}
\usepackage{amssymb}
\usepackage{amsmath}
\usepackage{mathrsfs}
\usepackage{amsthm}
\usepackage[colorlinks,linkcolor=blue,citecolor=blue]{hyperref}

 \newtheorem{theorem}{Theorem}[section]
\newtheorem{corollary}[theorem]{Corollary}
\newtheorem{lemma}[theorem]{Lemma}
\newtheorem{proposition}[theorem]{Proposition}

\theoremstyle{remark}
\newtheorem{remark}[theorem]{Remark}
\theoremstyle{definition}
\newtheorem{definition}[theorem]{Definition}

\numberwithin{equation}{section}

\begin{document}

\title[Hessian quotient equations]{The Dirichlet problem for Hessian quotient equations on exterior domains} 

\author[T.Y. Jiang]{Tangyu Jiang}
\address[T.Y. Jiang]{School of Mathematical Sciences\\
Beijing Normal University \\
	100875 Beijing\\
	P.R. China}
\email{202131081002@mail.bnu.edu.cn}

\author[H.G. Li]{Haigang Li}
\address[H.G. Li]{School of Mathematical Sciences\\
Beijing Normal University \\
	100875 Beijing\\
	P.R. China}
\email{hgli@bnu.edu.cn}

\author[X.L. Li]{Xiaoliang Li}
\address[X.L. Li]{School of Mathematical Sciences\\
Beijing Normal University \\
	100875 Beijing\\
	P.R. China}
\email{rubiklixiaoliang@163.com}



\begin{abstract}
In this paper, we consider the exterior Dirichlet problem for Hessian quotient equations with the right hand side $g$, where $g$ is a positive function and $g=1+O(|x|^{-\beta})$ near infinity, for some $\beta>2$. Under a prescribed generalized symmetric asymptotic behavior at infinity, we establish an existence and uniqueness theorem for viscosity solutions, by using comparison principles and Perron's method. This extends the previous results for Monge--Amp\`ere equations and Hessian equations.
\end{abstract}

\maketitle

\section{Introduction}\label{sec:intro}
The aim of this paper is to study the Dirichlet problem for Hessian quotient equation
\begin{equation}\label{eq:pro}
S_{k,l}(D^2u):=\frac{\sigma_k(\lambda(D^2u))}{\sigma_l(\lambda(D^2u))}=g(x)
\end{equation}
in the exterior domain $\mathbb{R}^n\setminus\overline{D}$, where $D$ is a bounded open set in $\mathbb{R}^n$ $(n\geq 3)$, $0\leq l<k\leq n$, $\lambda(D^2 u)=(\lambda_1,\cdots,\lambda_n)$ denotes the eigenvalue vector of the Hessian matrix $D^2u$, 
\begin{equation*}
\sigma_j(\lambda(D^2u))=
\begin{cases}
\sum_{1\leq i_1<\cdots<i_j\leq n}\lambda_{i_1}\cdots\lambda_{i_j}, & j=1,\cdots, n,\\
1,& j=0,
\end{cases}
\end{equation*}
and $g\in C^0(\mathbb{R}^n\setminus D)$ is a positive function satisfying 
\begin{equation}\label{eq:beta-g}
\inf_{\mathbb{R}^n\setminus D}g>0\quad\text{and}\quad\limsup_{|x|\to\infty}|x|^\beta|g(x)-1|<\infty
\end{equation}
for some constant $\beta>0$. We are mainly concerned with the existence and uniqueness of viscosity solutions to \eqref{eq:pro}, with prescribed boundary data and asymptotic behavior at infinity.

Note that equation \eqref{eq:pro} embraces several typical cases. When $l=0$, it is Poisson equation 
\begin{equation}\label{eq:Poisson}
\sigma_1(\lambda(D^2u))=\Delta u=g
\end{equation}
if $k=1$ and fully nonlinear $k$-Hessian equation 
\begin{equation}\label{eq:k-Hessian-g}
\sigma_k(\lambda(D^2u))=g
\end{equation}
if $2\leq k\leq n$, which in particular corresponds to the famous Monge--Amp\`ere equation 
\begin{equation}\label{eq:M-A-g}
\det(D^2u)=g
\end{equation}
for $k=n$. It is well-known that, as a special class of nonlinear, second-order elliptic equations of the form $f(\lambda(D^2u))=g(x)$, the classical solvability of the interior Dirichlet problem for equation \eqref{eq:pro} has been extensively studied. We would like to mention the work of Caffarelli--Nirenberg--Spruck \cite{Caffarelli1984,Caffarelli1985} on equations \eqref{eq:k-Hessian-g}--\eqref{eq:M-A-g}, and that of Trudinger \cite{Trudinger1995} on general equation \eqref{eq:pro}; see also Ivochkina \cite{Ivochkina1983,Ivochkina1985}, Urbas \cite{Urbas1990}, Guan \cite{Guan1994}, and the references therein.

In exterior domains, the study of the Dirichlet problem for equation \eqref{eq:pro} also has received increasing attentions in recent years. It goes back to Meyers--Serrin \cite{Serrin1960}, who gave a sufficient condition for the source term $g$ to obtain the existence and uniqueness of solutions to \eqref{eq:Poisson} with an assigned limit at infinity. For nonlinear case of \eqref{eq:pro}, it started by Caffarelli--Li \cite{Caffarelli-Li-2003}, which is an extension of celebrated J\"{o}rgens--Calabi--Pogorelov theorem \cite{Calabi1958,Jorgens1954,Pogorelov1972} stating that any classical convex solution of 
\begin{equation}\label{eq:M-A}
\det(D^2u)=1
\end{equation}
in $\mathbb{R}^n$ must be a quadratic polynomial (see also \cite{Caffarelli1995,Cheng-Yau-1986,Jost2001}). They showed that if $u$ is a convex viscosity solution of \eqref{eq:M-A} outside a bounded open convex subset of $\mathbb{R}^n$ ($n\geq3$), then there is a symmetric positive definite matrix $A\in\mathbb{R}^{n\times n}$ with $\det (A)=1$, a vector $b\in\mathbb{R}^n$ and a constant $c\in\mathbb{R}$ such that
\begin{equation}
\label{eq:C-Li}
\limsup_{|x|\to\infty}\left(|x|^{n-2}\left|u(x)-\left(\frac12x^TAx+b\cdot x+c\right)\right|\right)<\infty.
\end{equation}
Moreover, by Perron's method they established an existence and uniqueness theorem for exterior solutions of \eqref{eq:M-A} prescribed by \eqref{eq:C-Li}. Then Bao--Li--Zhang \cite{Bao-Li-Zhang-2015} extended that to general equation \eqref{eq:M-A-g}, assuming \eqref{eq:beta-g} with $\beta>2$. Recently, Li--Lu \cite{Li-Lu-2018} further obtained nonexistence results to the exterior problems considered in \cite{Bao-Li-Zhang-2015,Caffarelli-Li-2003}, giving a complete characterization for their solvability. 

In spirit of \cite{Caffarelli-Li-2003}, when $g\equiv 1$ the exterior Dirichlet problem for Hessian equations \eqref{eq:k-Hessian-g} and Hessian quotient equations \eqref{eq:pro} was subsequently treated by Bao--Li--Li \cite{Bao-Li-Li-2014} and Li--Li \cite{Li-Li-2018} respectively, under proper prescribed asymptotic condition similar to \eqref{eq:C-Li}. Inspired by \cite{Bao-Li-Zhang-2015}, Cao--Bao \cite{Cao-Bao-2017} further extended the result in \cite{Bao-Li-Li-2014} to equation \eqref{eq:k-Hessian-g} with a general right hand side $g$ satisfying \eqref{eq:beta-g}. However, equation \eqref{eq:pro} with general $g\not\equiv 1$ is much more complicated to deal with than that since whenever $l>0$ elementary symmetric functions $\sigma_k$ and $\sigma_l$, of different order of homogeneities, are involved simultaneously. In particular, in this case we are confronted with the effect of irreducible fraction involved in \eqref{eq:pro}, as will be explained in detail in Remark \ref{rk:thm} (ii) below. More importantly, the comparison principle for viscosity solutions of equation \eqref{eq:pro} have not been found in a general setting of the function $g$, as far as we know, so that ones are unable to carry out the Perron process to construct the solution of the Dirichlet problem as in the Monge--Amp\`ere case and $k$-Hessian case.   

Regarding more related studies on exterior domains, we refer the readers to \cite{Bao-Li-2013,Jiang-Li-Li-2020,Li-Bao-2014,Li2019} where other Hessian type equations associated with  \eqref{eq:pro}, including the special Lagrangian equations, are considered. Besides, the Liouville type results for equation \eqref{eq:pro} can be found in \cite{Bao2003,CY2010,Li-Li-Yuan-2019,Li-Ren-Wang-2016} and the references therein.

 In this paper, we shall prove the solvability of the exterior Dirichlet problem for equation \eqref{eq:pro}, under a prescribed quadratic condition of type \eqref{eq:C-Li}. Inspired by the discussions of Trudinger \cite{Trudinger1990} for the prescribed curvature equations, we investigate the comparison principles in a general setting of Hessian type fully nonlinear elliptic equations which may embrace \eqref{eq:pro} as a special case; see Appendix \ref{sec:comparison}. Then in the spirit of Crandall, Ishii, and Lions \cite{Ishii1992}, we are allowed to establish an existence and uniqueness theorem for viscosity solutions of \eqref{eq:pro} by applying an adapted Perron's method. This extends previously known results on Monge--Amp\`ere equation \eqref{eq:M-A-g} in \cite{Bao-Li-Zhang-2015} and on $k$-Hessian equation \eqref{eq:k-Hessian-g} in \cite{Cao-Bao-2017}, and also extends those obtained in \cite{Li-Li-2018} particularly for the case $g\equiv 1$ to the general $g$ fulfilling condition \eqref{eq:beta-g}. In order to state the result specifically, we introduce below some notations.

 Let $$\mathcal{A}_{k,l}=\left\{A\in S^+(n):S_{k,l}(A)=1\right\}$$ where $S^+(n)$ is the space of real $n\times n$ symmetric positive definite matrices. For $A\in S^+(n)$, we define
$$H_k(A)=\max_{1\leq i\leq n}\frac{\sigma_{k-1;i}(\lambda(A))\lambda_i(A)}{\sigma_k(\lambda(A))}$$
where $\sigma_{k-1;i}(\lambda):=\sigma_{k-1}(\lambda)|_{\lambda_i=0}$. Then $H_k(A)\in [\frac{k}{n},1)$ for $1\leq k\leq n-1$ and $H_n(A)\equiv 1$. Let $$\mathscr{A}_{k,l}=\left\{A\in\mathcal{A}_{k,l}: H_k(A)<\frac{k-l}{2}\right\}.$$ To avoid $\mathscr{A}_{k,l}$ to be empty, we require the indices $k$ and $l$ to satisfy
 \begin{equation}\label{eq:index-k-l}
\begin{cases}
0\leq l\leq n-3 &\text{if }k\geq l+2,\\
0\leq l<\frac{n}{2}-1 &\text{if }k=l+1.
\end{cases}
\end{equation}

Our main result is the following.
\begin{theorem}\label{thm:main}
Let $D$ be a smooth, bounded, strictly convex domain in $\mathbb{R}^n$ ($n\geq 3$) and let $\varphi\in C^2(\partial D)$. Assume that $g\in C^0(\mathbb{R}^n\setminus D)$ satisfies \eqref{eq:beta-g} with $\beta>2$. Then for any $A\in\mathscr{A}_{k,l}$ and $b\in\mathbb{R}^n$, there exists a constant $c_*$ depending only on $n,A,b,D,g$ and $\|\varphi\|_{C^2(\partial D)}$, such that for every $c>c_*$ there exists a unique viscosity solution $u\in C^0(\mathbb{R}^n\setminus D)$ to the problem 
\begin{equation}\label{eq:asym-thm}
\begin{cases}
S_{k,l}(D^2u)=g \quad\text{in}\ \mathbb{R}^n\setminus\overline{D},\\
u=\varphi \quad\text{on}\ \partial D,\\
\limsup_{|x|\to\infty}\left(E(x)\Big{|}u(x)-(\frac12x^TAx+b\cdot x+c)\Big{|}\right)<\infty,
\end{cases}
\end{equation}
where 
\begin{equation}\label{eq:intro-E}
E(x):=\left\{
\begin{array}{ll}
|x|^{\min\{\beta,\frac{k-l}{H_k(A)}\}-2} & \text{if }\beta\neq\frac{k-l}{H_k(A)},\\
|x|^{\frac{k-l}{H_k(A)}-2}(\ln |x|)^{-1} & \text{if }\beta=\frac{k-l}{H_k(A)}.  
\end{array}
\right.
\end{equation}
\end{theorem}

\begin{remark}\label{rk:thm}
(i) Notice that when $l=0$ and $2\leq k\leq n$, $\mathscr{A}_{k,0}=\mathcal{A}_{k,0}$ by Proposition \ref{G-pro:A-k-l} and Theorem \ref{thm:main} was proved in \cite{Bao-Li-Zhang-2015} for equation \eqref{eq:M-A-g} and in \cite{Cao-Bao-2017} for equation \eqref{eq:k-Hessian-g}, respectively. We here extend to the general case \eqref{eq:pro} and are able to prove Theorem \ref{thm:main} in a systematic way for all of $(l,k)$ satisfying \eqref{eq:index-k-l}. 

(ii) When $l\geq1$ and $g\equiv 1$, Theorem \ref{thm:main} was proved in \cite{Li-Li-2018} by rewriting \eqref{eq:pro} as $$\sigma_k(\lambda(D^2u))=\sigma_l(\lambda(D^2u)),$$ namely, not considering it in a sense of fraction. However, whenever $g\not\equiv 1$ one has to tackle the difficulty caused by the irreducible fraction in \eqref{eq:pro} when seeking subsolutions and supersolutions of \eqref{eq:pro} for carrying out the Perron process. Indeed, different from that adopted in \cite{Li-Li-2018}, here we turn to solve a new ODE to construct the desired subsolutions; see Remark \ref{G-rk:no-h-l} in Subsection \ref{G-sec:sub} for more details. Also, we can obtain the desired supersolutions in a parallel way, since unlike \cite{Li-Li-2018} we are unable to directly pick quadratic polynomials $\frac12x^TAx+b\cdot x+c$ as such ones unless assuming $\inf_{\mathbb{R}^n\setminus D}g=1$.
\end{remark}
\begin{remark}
It is not difficult to observe that Theorem \ref{thm:main} still holds with $\mathscr{A}_{k,l}$ adapted to a slightly larger set $
\mathscr{A}^*_{k,l}:=\{A\in\mathbb{R}^{n\times n}: A^*\in\mathscr{A}_{k,l}\}$
where $A^*=\frac{A+A^T}{2}$.
\end{remark}

The proof of Theorem \ref{thm:main} is based on Perron's method, formulated in Theorem \ref{thm:Perron-m} for the general fully nonlinear elliptic equation of form \eqref{app-eq:F-x-D2u} which includes \eqref{eq:pro} as a specific case. We remark that Theorem \ref{thm:Perron-m}  is an adaption of arguments as in \cite{Ishii1992,Ishii1989,Ishii-Lions-1990} once the uniqueness and comparison results are well established (see Theorem \ref{thm:comparison-b} and Corollary \ref{thm:comparison-ub}). With the help of that, the key ingredient of our proof lies in finding a family of appropriate subsolutions and supersolutions of \eqref{eq:pro} both with uniformly quadratic asymptotics at infinity. As a realization, we work with the so-called generalized symmetric functions (an extension of radial functions, specified in Definition \ref{G-def:G-Sym}), on which $k$-Hessian operators $S_{k,0}$ can be computed explicitly (see Lemma \ref{G-lem:k-He-G-Sym} below), thus helping us obtain the desired subsolutions and supersolutions by solving two corresponding second-order ODEs (\eqref{G-eq:ode-sub} and \eqref{G-eq:ode-super}). We would like to point out that the reason why we restrict $\beta>2$ and $A\in\mathscr{A}_{k,l}$ in Theorem \ref{thm:main} is to ensure the asymptotically quadratic property of subsolutions and supersolutions of \eqref{eq:pro} to be constructed in Propositions \ref{G-pro:sub} and \ref{G-pro:super}. One can see this from the computations performed in part \textbf{(c)} in Subsection \ref{G-sec:sub}.

The remainder of this paper is organized as follows. Section \ref{sec:G} is devoted to the construction of a family of subsolutions and supersolutions of \eqref{eq:pro} with asymptotically quadratic property. Then we apply Theorem \ref{thm:Perron-m} to prove Theorem \ref{thm:main} in Section \ref{sec:3}. In Appendix \ref{sec:comparison}, in a general setting of fully nonlinear elliptic equations related to the eigenvalues of the Hessian, we employ approximation arguments to prove comparison principles for viscosity solutions of the equations, both on bounded and unbounded domains. In Appendix \ref{sec:perron}, we state precisely Perron's method in Theorem \ref{thm:Perron-m} and present the proof, for the reader's convenience.

\section{Generalized symmetric subsolutions and supersolutions}\label{sec:G}
In this section, we first introduce the concept of generalized symmetric functions as well as their fine properties. Then by solving two second-order ODEs we obtain a family of generalized symmetric subsolutions and supersolutions of \eqref{eq:pro}, both of which are uniformly $k$-convex and asymptotically quadratic near infinity; see Propositions \ref{G-pro:sub} and \ref{G-pro:super}. This is essential to apply the Perron's arguments to prove Theorem \ref{thm:main} in the next section.

Throughout the section, we use the symbol $$\lambda(A):=(\lambda_1(A),\lambda_2(A),\cdots,\lambda_n(A))$$ with the order $\lambda_1(A)\leq\lambda_2(A)\leq\cdots\leq\lambda_n(A)$ to denote the eigenvalue vector of a real $n\times n$ symmetric matrix $A$.

\subsection{Preliminary}\label{G-sec:pre}
%For the reader's convenience, in this section 
We present some definitions and related facts which will be used later to seek proper subsolutions and supersolutions of  \eqref{eq:pro}.

As in \cite{Bao-Li-Li-2014}, we define the generalized symmetric (abbreviated to G-Sym in the sequel) function in the following sense.
\begin{definition}\label{G-def:G-Sym}
For a symmetric matrix $A$, we call $u$ a G-Sym function with respect to $A$ if it is a function of $s=\frac12x^TAx$, $x\in\mathbb{R}^n$, that is $$u(x)=u(s):=u(\frac12x^TAx).$$
If $u$ is a subsolution (supersolution) of \eqref{eq:pro} and is also a G-Sym function, we say that $u$ is a G-Sym subsolution (supersolution) of \eqref{eq:pro}.
\end{definition}

There is an explicit formula \eqref{G-eq:k-H-G-Sym} below for $k$-Hessian operators acting on smooth G-Sym functions and its proof can be found in Bao--Li--Li \cite{Bao-Li-Li-2014}. 
\begin{lemma}\label{G-lem:k-He-G-Sym}
Let $A=\mathrm{diag}(a_1,\cdots,a_n)$. If $w\in C^2(\mathbb{R}^n)$ is a G-Sym function with respect to $A$, then
\begin{equation}\label{G-eq:k-H-G-Sym}
\sigma_k(\lambda(D^2w))=\sigma_k(\lambda(A))(w')^k+w''(w')^{k-1}\sum_{i=1}^n\sigma_{k-1;i}(\lambda(A))(a_ix_i)^2,
\end{equation}
where $w':=\frac{dw}{ds}$ and $w'':=\frac{d^2w}{ds^2}$, $s=\frac12x^TAx$.
\end{lemma}

Taking advantage of the following properties of the $k$-th elementary symmetric function $\sigma_k$:
\begin{equation}\label{G-eq:sigma-k}
\sigma_k(\lambda)=\frac{1}{k}\sum_{i=1}^n\lambda_i\sigma_{k-1;i}(\lambda)=\sigma_{k;i}(\lambda)+\lambda_i\sigma_{k-1;i}(\lambda), \quad\forall\,
1\leq i\leq n,\forall\,\lambda\in\mathbb{R}^n,\\
\end{equation}
we can easily verify that
%\begin{Pro}\label{pro:H-k-A}
$$H_k(A)=\frac{\sigma_{k-1;n}(\lambda(A))\lambda_n(A)}{\sigma_k(\lambda(A))}$$ and
\begin{equation*}%\label{eq:H-k-A}
\frac{k}{n}\leq H_k(A)<1\text{ for }1\leq k\leq n-1; \ H_n(A)\equiv 1.
\end{equation*} 

We close this preliminary by presenting the relation of $\mathcal{A}_{k,l}$ and $\mathscr{A}_{k,l}$ as follows, which can be directly verified.
\begin{proposition}\label{G-pro:A-k-l}
Suppose $0\leq l<k\leq n$ and $n\geq 3$.
\begin{enumerate}
\item[(i)] $c^*(k,l)I\in\mathcal{A}_{k,l}$, where 
\begin{equation*}%\label{c*}
c^*(k,l):=\left(\frac{C_n^l}{C_n^k}\right)^{\frac{1}{k-l}},\quad C_n^j=\frac{n!}{(n-j)!j!},\  j=k,l.
\end{equation*}
\item[(ii)] If $k-l\geq 3$ or $k-l=2$ with $k<n$, then $\mathcal{A}_{k,l}=\mathscr{A}_{k,l}$.
\end{enumerate}
\end{proposition}

\subsection{Generalized symmetric subsolutions}\label{G-sec:sub}
We construct here the wanted G-Sym subsolutions of \eqref{eq:pro}. 

Throughout this subsection, let $A =\text{diag}(a_1, a_2,\cdots, a_n)\in\mathcal{A}_{k,l}$, $a:=\lambda(A)$ and $w:=w(s)$ is a G-Sym function with respect to $A$, where $s=\frac12x^TAx=\frac12\sum_{i=1}^na_ix_i^2$, $x\in\mathbb{R}^n$. We denote 
\begin{equation}\label{G-eq:D-s}
D(s):=\left\{x\in\mathbb{R}^n:\frac12x^TAx<s\right\}\quad\text{for }s>0.
\end{equation}

Since $g$ satisfies \eqref{eq:beta-g}, there exist $C_0$ and $s_0>1$ such that
\begin{equation}\label{G-eq:g-upper}
g(x)\leq \bar{g}(x)=\bar{g}(s):=1+C_0s^{-\frac{\beta}{2}}, \text{ when }s=\frac12 x^TAx\geq s_0.
\end{equation}
In order to make $w$ be a smooth subsolution of \eqref{eq:pro}, i.e. $S_{k,l}(D^2w)\geq g$, we consider the following ODE: 
\begin{equation}\label{G-eq:ode-sub}
\begin{cases}
(w')^{k-l}+2H_k(A)w''(w')^{k-l-1}s=\bar{g}(s),& s>1,\\
w'(s)>0,\quad w''(s)\leq 0,&s\geq1.
\end{cases}
\end{equation}
 A uniformly $k$-convex (see \eqref{app-eq:k-convex}) solution of \eqref{G-eq:ode-sub} will be a G-Sym subsolution of equation \eqref{eq:pro} with respect to $A$ when $s\geq s_0$. Indeed, applying Lemma \ref{G-lem:k-He-G-Sym}, one has 
\begin{align}
S_{k,l}(D^2w)&=\frac{\sigma_k(a)(w')^k+w''(w')^{k-1}\sum_{i=1}^n\sigma_{k-1;i}(a)(a_ix_i)^2}{\sigma_l(a)(w')^l+w''(w')^{l-1}\sum_{i=1}^n\sigma_{l-1;i}(a)(a_ix_i)^2}\notag\\
&\geq \frac{\sigma_k(a)(w')^k+w''(w')^{k-1}\sum_{i=1}^n\sigma_{k-1;i}(a)(a_ix_i)^2}{\sigma_l(a)(w')^l}\notag\\
&=(w')^{k-l}+w''(w')^{k-l-1}\sum_{i=1}^n\frac{\sigma_{k-1;i}(a)}{\sigma_l(a)}(a_ix_i)^2\notag\\
&\geq (w')^{k-l}+2H_k(A)w''(w')^{k-l-1}s=\bar{g}(s).\label{G-eq:S-D2w-sub}
\end{align}

For ODE \eqref{G-eq:ode-sub}, by the variation-of-constant formula we figure out it has a family of smooth solutions
\begin{equation}\label{G-eq:ode-sub-solu}
w_{c_1,c_2}(s):=c_2+\int_{s_0}^s\left(\eta^{-\mathcal{H}}\left(\int_1^\eta\mathcal{H}t^{\mathcal{H}-1}\bar{g}(t)\,dt+c_1\right)\right)^{\frac{1}{k-l}}d\eta
\end{equation}
where $c_i\in\mathbb{R}$ ($i=1,2$) are two constants and $$\mathcal{H}=\mathcal{H}(k,l,A):=\frac{k-l}{2H_k(A)}.$$ 
We next perform a series of careful calculations, which contains three parts, to see whether  $w_{c_1,c_2}$ given by \eqref{G-eq:ode-sub-solu} have the fine properties we expected.
 %and verify its solution's property including the $k$-convexity and specific asymptotic behavior at infinity.

\textbf{(a)} We claim $w_{c_1,c_2}'>0$ and $w_{c_1,c_2}''\leq 0$ on $[1,+\infty)$.
 
After a direct calculation, we have
$$w_{c_1,c_2}'(s)=\left(s^{-\mathcal{H}}\left(\int_1^s\mathcal{H}t^{\mathcal{H}-1}\bar{g}(t)\,dt+c_1\right)\right)^{\frac{1}{k-l}}>0$$
if $c_1>0$, and 
 \begin{align*}
w_{c_1,c_2}''(s)&=-\frac{s^{-\mathcal{H}-1}}{2H_k(A)}(w')^{1-k+l}\left(\int_1^s\mathcal{H}t^{\mathcal{H}-1}\bar{g}(t)\,dt+c_1-s^{\mathcal{H}}\bar{g}(s)\right)\\
&=:-\frac{s^{-\mathcal{H}-1}}{2H_k(A)}(w')^{1-k+l}G(s),
\end{align*}
where 
\begin{equation*}
G(s)=
\begin{cases}
c_1+\frac{C_0\beta}{2(\mathcal{H}-\frac{\beta}{2})}s^{\mathcal{H}-\frac{\beta}{2}}-1-\frac{C_0\mathcal{H}}{\mathcal{H}-\frac{\beta}{2}}&\text{if }\mathcal{H}\neq \frac{\beta}{2},\\
c_1+C_0\mathcal{H}\ln s-1-C_0&\text{if }\mathcal{H}=\frac{\beta}{2}.
\end{cases}
\end{equation*}
Hence, there is $\tilde{C}>0$, dependent on $\beta$ but independent of $s$, such that $w_{c_1,c_2}'(s)>0$ and $w_{c_1,c_2}''(s)<0$ for $s\geq1$ when $c_1>\tilde{C}$.

\textbf{(b)} Claim: $w_{c_1,c_2}(s)$ are uniformly $k$-convex on $[1,+\infty)$, provided $c_1>\tilde{C}$. 

For any $1\leq m\leq k$ and $c_1>\tilde{C}$,
\begin{align*}
&\sigma_m(\lambda(D^2 w_{c_1,c_2}))\\&=\sigma_m(a)(w_{c_1,c_2}')^m+w_{c_1,c_2}'' (w_{c_1,c_2}')^{m-1}\sum_{i=1}^n\sigma_{m-1;i}(a)(a_ix_i)^2\\
&=\sigma_m(a)(w_{c_1,c_2}')^{m-1}\left(w_{c_1,c_2}'+w_{c_1,c_2}''\sum_{i=1}^n\frac{\sigma_{m-1;i}(a)}{\sigma_m(a)}(a_ix_i)^2\right)\\
&\geq \sigma_m(a)(w_{c_1,c_2}')^{m-1}\left(w_{c_1,c_2}'+2H_m(A)w_{c_1,c_2}''s\right).
\end{align*}
Since
\begin{align*}
\sigma_m(a)\sigma_{k-1;n}(a)&=\left(\sigma_{m;n}(a)+a_n\sigma_{m-1;n}(a)\right)\sigma_{k-1;n}(a)\\
&\geq \sigma_{m-1;n}(a)\sigma_{k;n}(a)+a_n\sigma_{m-1;n}(a)\sigma_{k-1;n}(a)\\
&=\sigma_{m-1;n}(a)\sigma_{k}(a)
\end{align*}
which follows from \eqref{G-eq:sigma-k} and the Newtonian inequality (see \cite{Hardy1952}): 
$$\sigma_{k+1;n}(a)\sigma_{k-1;n}(a)\leq \sigma_{k;n}^2(a),\quad 1\leq k\leq n-1.$$
So that
$$H_m(A)=\frac{\sigma_{m-1;n}(a)a_n}{\sigma_m(a)}\leq \frac{\sigma_{k-1;n}(a)a_n}{\sigma_k(a)}=H_k(A)\quad \text{for }1\le m\le k.$$ 
Hence,
\begin{equation}\label{G-eq:k-convex-sub}
\sigma_m(\lambda(D^2w_{c_1,c_2}))\ge \sigma_m(a)(w_{c_1,c_2}')^{m-1}(w_{c_1,c_2}'+2H_k(A)w_{c_1,c_2}''s).\end{equation}
We claim $w_{c_1,c_2}'+2H_k(A)w_{c_1,c_2}''s>0$ for any $s\geq 1$. Indeed, from  \textbf{(a)} we obtain that
\begin{align*}
\frac{H_k(A)w_{c_1,c_2}''}{w_{c_1,c_2}'}&=-\frac{G(s)}{2(w_{c_1,c_2}')^{k-l}s^{\mathcal{H}+1}}\\
&=-\frac{1}{2s}\frac{\int_1^s\mathcal{H}t^{\mathcal{H}-1}\bar{g}(t)\,dt+c_1-s^{\mathcal{H}}\bar{g}(s)}{\int_1^s\mathcal{H}t^{\mathcal{H}-1}\bar{g}(t)\,dt+c_1}\\
&\geq -\frac{1}{2s}.
\end{align*}
 Consequently, \eqref{G-eq:k-convex-sub} shows $$\sigma_m(\lambda(D^2 w_{c_1,c_2}))>0, \quad 1\leq m\leq k.$$ Namely, $w_{c_1,c_2}$ are uniformly $k$-convex when $c_1>\tilde{C}$.

\textbf{(c)} We determine the asymptotic behavior of $w_{c_1,c_2}(s)$ as $s$ tends to infinity. 

We start by showing that $w_{c_1,c_2}'$ is close to $1$ at infinity in both cases $\mathcal{H}\neq\frac{\beta}{2}$ and $\mathcal{H}=\frac{\beta}{2}$. When $\mathcal{H}\neq\frac{\beta}{2}$, we have
\begin{align}
w_{c_1,c_2}'(s)-1=&\left(s^{-\mathcal{H}}\left(t^{\mathcal{H}}\Big|^s_1+\frac{2C_0\mathcal{H}}{2\mathcal{H}-\beta}t^{\mathcal{H}-\frac{\beta}{2}}\Big|^s_1+c_1\right)\right)^{\frac{1}{k-l}}-1\notag\\
=&\,O\left(s^{-\min\{\frac{\beta}{2},\mathcal{H}\}}\right), \quad \text{as }s\rightarrow\infty.\label{G-eq:O-w'-sub-1}
\end{align}
When $\mathcal{H}=\frac{\beta}{2}$, 
\begin{align}
w_{c_1,c_2}'(s)-1=&\left(s^{-\mathcal{H}}\left((1+C_0t^{\frac{\beta}{2}})t^{\mathcal{H}}\Big|^s_1+\frac{\beta}{2}\ln t\Big|^s_1+c_1\right)\right)^{\frac{1}{k-l}}-1\notag\\
=&O\left(s^{-\mathcal{H}}\ln s\right), \quad \text{as }s\rightarrow\infty.\label{G-eq:O-w'-sub-2}
\end{align}

We now rewrite $w_{c_1,c_2}(s)$ as:
\begin{align*}
w_{c_1,c_2}(s)&=c_2+\int_{s_0}^sw_{c_1,c_2}'(\eta)\,d\eta\\
&=c_2+s-s_0+\int_{s_0}^s\left(w_{c_1,c_2}'(\eta)-1\right)d\eta\\
&=s+\mu(c_1,c_2)-\int_{s}^\infty\left(w_{c_1,c_2}'(\eta)-1\right)d\eta
\end{align*}
where
\begin{equation}\label{G-eq:mu-c1}
\mu(c_1,c_2)=c_2-s_0+\int_{s_0}^\infty\left(w_{c_1,c_2}'(\eta)-1\right)d\eta.
\end{equation}
By \eqref{G-eq:O-w'-sub-1} and \eqref{G-eq:O-w'-sub-2}, $\mu(c_1,c_2)<\infty$ only if $\frac{\beta}{2}>1$ and $ \mathcal{H}>1$, and in this situation,
\begin{equation}\label{G-eq:asym-sub}
\begin{cases}
w_{c_1,c_2}(s)=s+\mu(c_1,c_2)+O(s^{1-\min\{\frac{\beta}{2},\mathcal{H}\}})&\text{if }\mathcal{H}\neq\frac{\beta}{2},\\
w_{c_1,c_2}(s)=s+\mu(c_1,c_2)+O(s^{1-\mathcal{H}}\ln s)&\text{if }\mathcal{H}=\frac{\beta}{2}.
\end{cases}
\end{equation}
This indicates when $\beta>2$ and $A\in\mathscr{A}_{k,l}$, $w_{c_1,c_2}$ given by \eqref{G-eq:ode-sub-solu} is asymptotically close to a quadratic polynomial at infinity.

Based on \textbf{(a)}--\textbf{(c)}, we conclude that
\begin{lemma}\label{G-lem:ode-sub-solu}
Assume $A\in\mathscr{A}_{k,l}$ and $\beta>2$. Then there exists $\tilde{C}>1$ dependent of $A$ and $\beta$, such that when $c_1>\tilde{C}$ the G-Sym function $w_{c_1,c_2}(s)$ given by \eqref{G-eq:ode-sub-solu} is a uniformly $k$-convex solution of problem \eqref{G-eq:ode-sub} and satisfies \eqref{G-eq:asym-sub}.
\end{lemma}

As a consequence of Lemma \ref{G-lem:ode-sub-solu}, we arrive at the following result.

\begin{proposition}\label{G-pro:sub}
Assume that $A\in\mathscr{A}_{k,l}$ is diagonal and $g$ satisfies \eqref{eq:beta-g} for some $\beta>2$. Let $s_0$ be as in \eqref{G-eq:g-upper} and $\tilde{C}$ be given in Lemma \ref{G-lem:ode-sub-solu}. Then when $c_1>\tilde{C}$ the G-Sym function $w_{c_1,c_2}(x)=w_{c_1,c_2}(s)$ given by \eqref{G-eq:ode-sub-solu} is a uniformly $k$-convex subsolution of equation \eqref{eq:pro}  in $\mathbb{R}^n\setminus D(s_0)$ and fulfills
\begin{equation*}
w_{c_1,c_2}(x)=\frac12x^TAx+\mu(c_1,c_2)+O\left(|x|^{2-\min\{\beta, \frac{k-l}{H_k(A)}\}}\right)\quad\text{as }|x|\to+\infty
\end{equation*} 
if $\beta\neq\frac{k-l}{H_k(A)}$, or 
 \begin{equation*}
w_{c_1,c_2}(x)=\frac12x^TAx+\mu(c_1,c_2)+O\left(|x|^{2-\frac{k-l}{H_k(A)}}\ln |x|\right)\quad\text{as }|x|\to+\infty
\end{equation*}
if $\beta=\frac{k-l}{H_k(A)}$, where $\mu(c_1,c_2)$ is as in \eqref{G-eq:mu-c1}.
\end{proposition}
\begin{proof}
It follows directly by combining \eqref{G-eq:g-upper}, \eqref{G-eq:S-D2w-sub} and Lemma \ref{G-lem:ode-sub-solu}.
\end{proof}

\begin{remark}\label{G-rk:no-h-l}
 We would like to make some comparisons of Proposition \ref{G-pro:sub} with related results available in literature. For $2\leq k\leq n$ and $l=0$, the G-Sym subsolutions given by Proposition \ref{G-pro:sub} also are previously constructed in \cite{Bao-Li-Zhang-2015,Cao-Bao-2017} for equations \eqref{eq:k-Hessian-g} and \eqref{eq:M-A-g}, while for $l>0$ and $g\equiv 1$ they are different from those obtained in \cite{Li-Li-2018} for equation \eqref{eq:pro} due to the ODE \eqref{G-eq:ode-sub} satisfied by subsolutions disagrees that derived in \cite{Li-Li-2018}. Actually, if $g\equiv1$, then \eqref{eq:pro} becomes $\sigma_{k}(\lambda(D^2u))=\sigma_l(\lambda(D^2u))$, from which via \eqref{G-eq:k-H-G-Sym} one can study the following ODE
\begin{equation}\label{G-eq:ode-sub-1}
(w')^{k}+2H_k(A)w''(w')^{k-1}s=(w')^{l-l}+2h_l(A)w''(w')^{l-1}s
\end{equation}
to seek subsolutions of \eqref{eq:pro}, where $h_l(A)\in[0,\frac{l}{n}]$ is defined by $$h_l(A)=\min_{1\leq i\leq n}\frac{\sigma_{l-1;i}(\lambda(A))\lambda_i(A)}{\sigma_l(\lambda(A))}.$$ Here introducing the quantities $H_k(A)$ and $h_l(A)$ at the same time is to strike a balance between different order of homogeneities and then \eqref{G-eq:ode-sub-1} could be analyzed explicitly as in \cite{Li-Li-2018}. However, whenever $g\not\equiv1$, the same idea is not helpful for removing the effect of irreducible fraction in \eqref{eq:pro}. In order to treat this, here we drop part of the denominator $\sigma_l(\lambda(D^2w))$ as handled in \eqref{G-eq:S-D2w-sub} (equivalent to putting $h_l(A)=0$ in \eqref{G-eq:ode-sub-1}). As a result, our admissible set of $A$, $\mathscr{A}_{k,l}$, may be a subset of the one in \cite{Li-Li-2018} which is given by
$$\tilde{\mathscr{A}}_{k,l}:=\left\{A\in\mathcal{A}_{k,l}: H_k(A)-h_l(A)<\frac{k-l}{2}\right\}.$$
Nevertheless, by Proposition \ref{G-pro:A-k-l} our method still works well in cases $k-l\geq2$ where it holds $$\mathscr{A}_{k,l}=\tilde{\mathscr{A}}_{k,l}=\mathcal{A}_{k,l}.$$
\end{remark}

\subsection{Generalized symmetric supersolutions}\label{G-sec:sup}
We will construct a family of G-Sym supersolutions of \eqref{eq:pro} coinciding at infinity with the G-Sym subsolutions we just obtained in Proposition \ref{G-pro:sub}, which is necessary in Perron's construction (Theorem \ref{thm:Perron-m}). Unlike the well-treated case $g\equiv 1$ in which one can directly pick $\frac12x^TAx+c$ to be desired supersolutions (see for example \cite{Bao-Li-Li-2014,Caffarelli-Li-2003,Jiang-Li-Li-2020,Li-Li-2018,Li-Bao-2014}), here we adapt the idea in Subsection \ref{G-sec:sub} to study the ODE:  
\begin{equation}\label{G-eq:ode-super}
\begin{cases}
(w')^{k-l}+2H_k(A)w''(w')^{k-l-1}s=\underline{g}(s),& s>1,\\
w'(s)>0, \quad w''(s)\geq0,&s\geq1,
\end{cases}
\end{equation}
where $\underline{g}$ is an increasing smooth function of $s$ satisfying, without loss of generality,
\begin{equation}\label{G-eq:g-lower}
0<\underline{g}\leq g\text{ for}\ s\geq 1,\quad\text{and}\quad\underline{g}=1-C_0s^{-\frac{\beta}{2}}\ \text{when }s\geq s_0,
\end{equation}
where $C_0$ and $s_0$ are as in \eqref{G-eq:g-upper}. Replacing $\bar{g}$ by $\underline{g}$ in \eqref{G-eq:ode-sub-solu} and taking $c_1=0$, one can easily obtain the following uniformly $k$-convex solutions of \eqref{G-eq:ode-super}:
\begin{equation}\label{G-eq:ode-super-solu}
\overline{w}_{c_2}(s):=c_2+\int_{1}^s\left(\eta^{-\mathcal{H}}\left(\int_1^\eta\mathcal{H}t^{\mathcal{H}-1}\underline{g}(t)\,dt\right)\right)^{\frac{1}{k-l}}d\eta
\end{equation}
Moreover, thanks to \eqref{G-eq:g-lower}, arguing as in part \textbf{(c)} before we infer that $\overline{w}_{c_2}$ possess asymptotics \eqref{G-eq:asym-sub} where $\mu(c_1,c_2)$ need to be replaced by
\begin{equation}\label{G-eq:lmu-c1-c2}
\overline{\mu}(c_2)=c_2-1+\int_{1}^\infty\left(\overline{w}_{c_2}'(\eta)-1\right)d\eta.
\end{equation}

Similarly as showed in \eqref{G-eq:S-D2w-sub}, we can find for $s\geq 1$, $$S_{k,l}(D^2\overline{w}_{c_2})\leq (\overline{w}_{c_2}')^{k-l}+2H_k(A)\overline{w}_{c_2}''(\overline{w}_{c_2}')^{k-l-1}s=\underline{g}(s)\leq g.$$
Namely, $\overline{w}_{c_2}$ are supersolutions of \eqref{eq:pro}.

The following statement is a summary of above facts, which serves as a counterpart of Proposition \ref{G-pro:sub}.
\begin{proposition}\label{G-pro:super}
Assume that $A\in\mathscr{A}_{k,l}$ is diagonal and $g$ satisfies \eqref{eq:beta-g} for some $\beta>2$. Let $c_2\in\mathbb{R}$ and $s=\frac12x^TAx$. Then the G-Sym function $\overline{w}_{c_2}(x)=\overline{w}_{c_2}(s)$ given by \eqref{G-eq:ode-super-solu} is a uniformly $k$-convex supersolution of equation \eqref{eq:pro} in $\mathbb{R}^n\setminus D(1)$ and fulfills
\begin{equation*}
\overline{w}_{c_2}(x)=\frac12x^TAx+\overline{\mu}(c_2)+O\left(|x|^{2-\min\{\beta, \frac{k-l}{H_k(A)}\}}\right)\quad\text{as }|x|\to+\infty
\end{equation*} 
if $\beta\neq\frac{k-l}{H_k(A)}$, or 
 \begin{equation*}
\overline{w}_{c_2}(x)=\frac12x^TAx+\overline{\mu}(c_2)+O\left(|x|^{2-\frac{k-l}{H_k(A)}}\ln |x|\right)\quad\text{as } |x|\to+\infty
\end{equation*}
if $\beta=\frac{k-l}{H_k(A)}$, where $\overline{\mu}(c_2)$ is as in \eqref{G-eq:lmu-c1-c2}.
\end{proposition}

\begin{remark}
One may wonder whether there are smooth G-Sym solutions of \eqref{eq:pro}, provided that $g$ is of G-Sym. Generally, this cannot be expected for all $0\leq l<k\leq n$. When $g\equiv 1$, \cite[Proposition 1.1]{Li-Li-Zhao-2019} states that \eqref{eq:pro} admits a G-Sym solution of $C^2$ with respect to a diagonal $A\in\mathcal{A}_{k,l}$ if and only if $l=0$ and $k=n$, unless $A=c^*(k,l)I$ (see (i) of Proposition \ref{G-pro:A-k-l}). The same rigidity result was also exploited in \cite{Cao-Bao-2017} for \eqref{eq:k-Hessian-g} with $g=\bar{g}$, introduced in \eqref{G-eq:g-upper}. That is the reason why we here look for G-Sym subsolutions and G-Sym supersolutions of \eqref{eq:pro} separately.
\end{remark}

\section{Proof of Theorem \ref{thm:main}} \label{sec:3}
In view of Theorem \ref{thm:Perron-m}, to prove Theorem \ref{thm:main} it suffices to demonstrate the existence of a viscosity subsolution $\underline{u}$ of equation \eqref{eq:pro} attaining the prescribed boundary value and asymptotic behavior at infinity, as well as that of a viscosity supersolution $\bar{u}\geq\underline{u}$ but agreeing on $\underline{u}$ at infinity. Roughly speaking, such a subsolution can be obtained by splicing together the supremum of barriers over the boundary points of the domain and the G-Sym subsolution constructed in Proposition \ref{G-pro:sub}; the desired supersolution is prepared in Proposition \ref{G-pro:super}. The uniqueness of the solution is guaranteed by comparison principle, Corollary \ref{thm:comparison-ub}.

Set $\mathcal{G}=\sup_{\mathbb{R}^n\setminus D} g$. To process the boundary behavior of the solution, we need the following existence result of barrier functions.  

\begin{lemma}\label{P-lem:w-xi}
Let $D$ be a bounded strictly convex domain of $\mathbb{R}^n$ ($n\geq3$) with $\partial D\in C^2$ and let $\varphi\in C^2(\partial D)$. For an invertible and symmetric matrix $A$, there exists some constant $C$, depending only on $n,\mathcal{G},\|\varphi\|_{C^2(\partial D)}$, the upper bound of $A$, the diameter and the convexity of $D$, and the $C^2$ norm of $\partial D$, such that for every $\xi\in\partial D$, there exists $\bar{x}(\xi)\in\mathbb{R}^n$ satisfying $$|\bar{x}(\xi)|\leq C\quad\text{and}\quad w_\xi<\varphi\quad\text{on}\quad \partial D\setminus\{\xi\},$$ where $$w_\xi(x)=\varphi(\xi)+\frac{\mathcal{G}^{\frac{1}{k-l}}}{2}\left[(x-\bar{x}(\xi))^TA(x-\bar{x}(\xi))-(\xi-\bar{x}(\xi))^TA(\xi-\bar{x}(\xi))\right]$$
for $x\in\mathbb{R}^n$.
\end{lemma}

\begin{proof}
As proved in \cite{Cao-Bao-2017} for $A\in\mathcal{A}_{k,0}$, it is a direct adaption of the arguments as in the proofs of \cite[Lemma 5.1]{Caffarelli-Li-2003} and \cite[Lemma 3.1]{Bao-Li-Li-2014} for the case $\mathcal{G}=1$. We thus omit it. 
\end{proof}
%We now start to prove Theorem \ref{thm:main}, provided that $A\in\mathcal{A}$ is of the form $$A=\text{diag}(a_1,a_2,\cdots,a_n)$$ 
%with $0<a_1\leq a_2\leq\cdots\leq a_n$, and that $b=0$.
We now present the proof of Theorem \ref{thm:main}. 

\begin{proof}[Proof of Theorem \ref{thm:main}]
By an orthogonal transformation and by subtracting a linear function from $u$, we need only to prove for $A=\text{diag}(a_1,a_2,\cdots,a_n)$ and $b=0$; see \cite[Lemma 3.3]{Li-Li-2018} for a specific demonstration. Also, without loss of generality, we assume $D(1)\subset D\subset D(s_0)$, where $D(\cdot)$ is as in \eqref{G-eq:D-s} and $s_0$ is as in \eqref{G-eq:g-upper}. 

\textbf{Step 1} We first construct a viscosity subsolution $\underline{u}$ of \eqref{eq:pro} with $\underline{u}=\varphi$ on $\partial D$ and the asymptotics as in \eqref{eq:asym-thm}.

Recalling $w_{c_1,c_2}(s)$ given in \eqref{G-eq:ode-sub-solu}, by Proposition \ref{G-pro:sub} it is a uniformly $k$-convex subsolution of \eqref{eq:pro} satisfying
\begin{equation}\label{P-eq:asym-E}
w_{c_1,c_2}(x)=\frac12x^TAx+\mu(c_1,c_2)+O\left(E^{-1}(x)\right),\quad |x|\to\infty,
\end{equation}
 provided $c_1>\tilde{C}$ and $x\in\mathbb{R}^n\setminus D(s_0)$. Here $E(x)$ is defined in \eqref{eq:intro-E}. We later will pick suitable $c_1$ to make \eqref{P-eq:asym-E} reach asymptotics \eqref{eq:asym-thm} as in \eqref{P-eq:mu-c1-c} below. Now regarding the boundary value, we are going to find another viscosity subsolution $\underline{w}$ attaining $\varphi$ on $\partial D$, such that, for some fixed $\bar{s}>s_0$, 
\begin{equation}\label{P-eq:barrier}
\max_{\partial D(s_0)}w_{c_1,c_2}\leq\min_{\partial D(s_0)}\underline{w}\quad\text{and}\quad\min_{\partial D(\bar{s})}w_{c_1,c_2}\geq\max_{\partial D(\bar{s})}\underline{w}.
\end{equation} 
For this aim, we set $$\underline{w}(x)=\max\{w_{\xi}(x)~|~\xi \in \partial D\},$$ where $w_\xi$ is introduced in Lemma \ref{P-lem:w-xi}. Clearly, $\underline{w}=\varphi$ on $\partial D$.  Since $$S_{k,l}(D^2w_\xi)=\mathcal{G}S_{k,l}(A)=\mathcal{G}\geq g\quad\text{in }\mathbb{R}^n\setminus\overline{D},$$
$w_\xi$ is a smooth convex subsolution of \eqref{eq:pro}. By Lemma \ref{lem:Perron-sub}, $\underline{w}$ is also a viscosity subsolution of \eqref{eq:pro}. Let 
$$c_2=m_{s_0}:=\min\{w_{\xi}(x)~|~\xi\in\partial D, x\in \overline{D(s_0)}\setminus D\}.$$
Thus, by \eqref{G-eq:ode-sub-solu} $w_{c_1,m_{s_0}}$ satisfies the first condition in \eqref{P-eq:barrier}. To realize the second one, it suffices to choose a large $c_1$ (assume $c_1\geq\alpha>\tilde{C}$) since $w_{c_1,m_{s_0}}$ is monotonically increasing with respect to $c_1$.

We next fix $c_*$ such that given $c>c_*$ one can find $c_1(c)$ to fulfill 
\begin{equation}\label{P-eq:mu-c1-c}
\mu(c_1(c),m_{s_0})=c.
\end{equation}
Notice from \eqref{G-eq:mu-c1} that $\mu(c_1,m_{s_0})$ is strictly increasing in $c_1$, and $$\lim_{c_1\to+\infty}\mu(c_1,m_{s_0})=+\infty.$$ Hence, if $c_*\geq\mu(\alpha,m_{s_0})$, then such $c_1(c)>\alpha$ exists. This means that, for $c>c_*$, $w_{c_1(c),m_{s_0}}$ satisfies \eqref{P-eq:barrier} and possesses asymptotic behavior:
\begin{equation}\label{P-eq:asym-sub}
w_{c_1(c),m_{s_0}}(x)=\frac12x^TAx+c+O\left(E^{-1}(x)\right),\quad |x|\to\infty.
\end{equation}

For $c>c_*$, we define
\begin{equation*}
\underline{u}(x)=
\begin{cases}
\underline{w}(x),\quad &x\in D(s_0)\backslash D,\\
\max\{w_{c_1(c),m_{s_0}}(x),\underline{w}(x)\},\quad &x\in D(\bar{s})\backslash D(s_0),\\
w_{c_1(c),m_{s_0}}(x),&x\in\mathbb{R}^n\backslash D(\bar{s}).
\end{cases}
\end{equation*}
Then from Definition \ref{app-def:visc} and Lemma \ref{lem:Perron-sub}, $\underline{u}$ is a viscosity subsolution of \eqref{eq:pro} satisfying \eqref{P-eq:asym-sub}, and $\underline{u}=\underline{w}=\varphi$ on $\partial D$.

\textbf{Step 2}
We construct a viscosity supersolution $\bar{u}$ of \eqref{eq:pro} to satisfy 
\begin{equation}\label{P-eq:step-super}
\underline{u}\leq\bar{u}\ \text{ in }\mathbb{R}^n\setminus\overline{D}\quad\text{and}\quad\lim_{|x|\to\infty}(\bar{u}-\underline{u})(x)=0.
\end{equation}

  By Proposition \ref{G-pro:super}, $\overline{w}_{c_2}$ given in \eqref{G-eq:ode-super-solu} is a uniformly $k$-convex supersolution of \eqref{eq:pro} satisfying
\begin{equation}\label{P-eq:asym-super}
\overline{w}_{c_2}(x)=\frac{1}{2}x^TAx+\overline{\mu}(c_2)+O\left(E^{-1}(x)\right),\quad |x|\to\infty.
\end{equation}
Then for $c>c_*$ the second condition in \eqref{P-eq:step-super} holds for $\overline{w}_{c_2(c)}$ where $c_2(c)$ is determined by $$\overline{\mu}(c_2(c))=c,$$ which implies \eqref{P-eq:asym-super} agrees with \eqref{P-eq:asym-sub}. Actually, by definition of $\overline{\mu}$ (see \eqref{G-eq:lmu-c1-c2}), $$c_2(c)=c+1-\int_1^\infty\left[\left(s^{-\mathcal{H}}\int_1^s\mathcal{H}t^{\mathcal{H}-1}\underline{g}(t)\,dt\right)^{\frac{1}{k-l}}-1\right]\,ds\geq c\,;$$
here we have used the fact that $\underline{g}\leq 1$ for $s\geq 1$. Next, with the help of comparison principles proved in Appendix \ref{sec:comparison}, we show that $\overline{w}_{c_2(c)}$ also agrees the first condition in \eqref{P-eq:step-super} for proper $c$.

In \textbf{Step 1}, we fixed $c_*\geq\mu(\alpha,m_{s_0})$. We now further require that $c_*>M_{s_0}$, where 
$$M_{s_0}:=\max\{w_{\xi}(x)~|~\xi\in\partial D, x\in \overline{D(s_0)}\setminus D\}.$$
Then for $c>c_*$, 
$$\overline{w}_{c_2(c)}\geq c_2(c)\geq c>c_*>M_{s_0}\geq m_{s_0}\geq w_{c_1(c),m_{s_0}}\quad\text{on }\partial D(s_0),$$
and also $$\lim_{|x|\to\infty}(\overline{w}_{c_2(c)}-w_{c_1(c),m_{s_0}})(x)=0.$$
Applying Corollary \ref{thm:comparison-ub}, we thus deduce that 
\begin{equation}\label{P-eq:compare-super-sub}
\overline{w}_{c_2(c)}\geq w_{c_1(c),m_{s_0}}\quad\text{in }\mathbb{R}^n\setminus\overline{D(s_0)}.
\end{equation}
On the other hand, 
$$\overline{w}_{c_2(c)}\geq c_2(c)\geq c>c_*>M_{s_0}\geq \underline{w}\quad\text{on }\partial D,$$
and by \eqref{P-eq:barrier}, $$\overline{w}_{c_2(c)}\geq w_{c_1(c),m_{s_0}}\geq\underline{w}\quad\text{on }\partial D(\bar{s}).$$
Hence, applying Theorem \ref{thm:comparison-b} yields 
\begin{equation}\label{P-eq:compare-super-barrier} 
\overline{w}_{c_2(c)}\geq\underline{w}\quad\text{in }D(\bar{s})\setminus\overline{D}.
\end{equation}
By virtue of \eqref{P-eq:compare-super-sub} and \eqref{P-eq:compare-super-barrier}, we get $$\overline{w}_{c_2(c)}\geq \underline{u}\quad\text{in }\mathbb{R}^n\setminus\overline{D}$$ for $c>c_*$.

Based on the above arguments, we let $\bar{u}=\overline{w}_{c_2(c)}$ for $c>c_*$, which is a desired viscosity supersolution of \eqref{eq:pro}.

\textbf{Step 3} We show the existence and uniqueness of viscosity solutions to problem \eqref{eq:asym-thm}.

With $\underline{u}$ and $\bar{u}$, we define 
\begin{align*}
u(x):=&\sup\{v(x)| v\in\mathrm{USC}(\mathbb{R}^n\setminus\overline{D}), S_{k,l}(D^2v)\geq g\text{ in }\mathbb{R}^n\setminus\overline{D}\text{ in the}\\
&\text{viscosity sense, with }\underline{u}\leq v\leq\bar{u}\text{ in }\mathbb{R}^n\setminus\overline{D}\text{ and }v=\varphi\text{ on }\partial D\}.
\end{align*}
Since $\underline{u}$ and $\bar{u}$ both satisfy \eqref{P-eq:asym-sub}, one has
$$\limsup_{|x|\to\infty}\left(E(x)\Big{|}u(x)-(\frac12x^TAx+c)\Big{|}\right)<\infty.$$
From Theorem \ref{thm:Perron-m}, we thus conclude that $u\in C^0(\mathbb{R}^n\setminus D)$ is a viscosity solution of \eqref{eq:asym-thm}. 

Finally, the uniqueness of viscosity solutions to problem \eqref{eq:asym-thm} is forced by Corollary \ref{thm:comparison-ub}. This completes the proof. 
\end{proof}

\begin{remark}
We remark that from the above demonstration the lowerbound $c_*$ of $c$ in Theorem \ref{thm:main} may not be removed. Actually, in \cite[Theorem 1.3]{Li-Lu-2018}, Li and Lu proved that when $l=0$ and $k=n$, there is a sharp constant $c_*$ such that problem \eqref{eq:asym-thm} admits a viscosity solution if and only if $c\geq c_*$. It would be interesting to see whether 
such a sharp charaterization result holds for equation \eqref{eq:pro} with general $0\leq l<k\leq n$ and will be left to study in our future work.
\end{remark}

\appendix
\section{Comparison principles for viscosity solutions}\label{sec:comparison}
In this appendix, we prove comparison principles for viscosity subsolutions and viscosity supersolutions of fully nonlinear, second-order partial differential equations of the form
\begin{equation}\label{app-eq:F-x-D2u}
F(x,D^2u):=f(\lambda(D^2u))-g(x)=0\quad\text{in }\Omega,
\end{equation}
where $\Omega$ is an open subset of $\mathbb{R}^n$, $f$ is a symmetric function of $C^1$ defined on an open convex symmetric cone $\Gamma$ in $\mathbb{R}^n$, with vertex at the origin and containing the positive cone $\Gamma^+$, $\lambda(D^2u)=(\lambda_1,\cdots,\lambda_n)$ denotes the eigenvalue vector of the Hessian matrix $D^2u$, and $g\in C^0(\Omega)$ is a given function with $\inf_{\Omega}g>0$. Throughout the section, we assume that 
\begin{equation}\label{app-eq:f-increase}
\frac{\partial f}{\partial\lambda_i}>0\quad \text{on }\Gamma,\quad i=1,\cdots,n, 
\end{equation} 
by which $F(x,D^2u)$ in \eqref{app-eq:F-x-D2u} is an elliptic operator on those functions $u\in C^2(\Omega)$ such that $\lambda(D^2u)\in\Gamma$. Such functions will be called admissible. In addition, it is assumed 
\begin{equation}\label{app-eq:f-nu}
\sum_{i=1}^n\lambda_i\frac{\partial f}{\partial\lambda_i}\geq \nu(f) \quad\text{on }\Gamma
\end{equation}
and 
\begin{equation}\label{app-eq:f-boundary}
\limsup_{\lambda\to\lambda_0} f(\lambda)<\inf_{\Omega} g \quad\text{for }\forall\,\lambda_0\in\partial\Gamma,
\end{equation}
where $\nu$ is some positive increasing function on $\mathbb{R}$. Clearly, the quotient of elementary symmetric functions $\sigma_k/\sigma_l$ $(0\leq l<k\leq n)$ involved in \eqref{eq:pro} is an example of $f$ fulfilling \eqref{app-eq:f-increase}--\eqref{app-eq:f-boundary} if we define
\begin{equation}\label{app-eq:k-convex}
\Gamma=\Gamma_k:=\{\lambda\in\mathbb{R}^n~|~\sigma_j(\lambda)>0,\ 1\leq j\leq k\}
\end{equation}
(admissible functions are called \emph{uniformly $k$-convex} in this case); indeed, for this example the validity of \eqref{app-eq:f-increase} as well as that of \eqref{app-eq:f-boundary} is well-known (see for instance \cite{Caffarelli1985,Trudinger1995}), and \eqref{app-eq:f-nu} holds obviously via the homogeneity.

We first recall the definition of the viscosity solution to equation \eqref{app-eq:F-x-D2u} following \cite{Caffarelli-Cabre-1995,Ishii1992,Urbas1990}. For simplicity, let $\mathrm{USC}(\Omega)$ and $\mathrm{LSC}(\Omega)$ respectively denote the set of upper and lower semicontinuous real valued functions on $\Omega$; let $B(\Omega)$ be the set of bounded functions on $\Omega$, and also let $B_{p}(\Omega)$ be the set of functions that are bounded in the intersection of $\Omega$ and any ball of $\mathbb{R}^n$.

\begin{definition}\label{app-def:visc}
A function $u\in\mathrm{USC}(\Omega)$ $(\mathrm{LSC}(\Omega))$ is said to be a viscosity subsolution (supersolution) of \eqref{app-eq:F-x-D2u} (or say that $u$ satisfies $F(x,D^2u)\geq(\leq)0$ in the viscosity sense), if for any open subset $A$ of $\Omega$, any admissible function $\psi\in C^2(A)$, and any local maximum (minimum) $x_0\in A$ of $u-\psi$ we have $$F(x_0,D^2\psi(x_0))\geq(\leq)0.$$
A function $u\in C^0(\Omega)$ is said to be a viscosity solution of \eqref{app-eq:F-x-D2u} if it is both a viscosity subsolution and a viscosity supersolution of \eqref{app-eq:F-x-D2u}.
\end{definition}

\begin{remark}\label{app-rk:connection}
 An admissible solution of \eqref{app-eq:F-x-D2u} is clearly a viscosity solution. Conversely, as argued in \cite[Proposition 2.2]{Urbas1990} for $k$-Hessian equations, under assumption \eqref{app-eq:f-boundary} one would see that a viscosity subsolution of \eqref{app-eq:F-x-D2u} is admissible at each point at which it is twice differentiable. %Hence, if $u\in C^2(\Omega)$ is a viscosity solution of \eqref{app-eq:F-x-D2u} with \eqref{app-eq:f-boundary} holding, then $u$ is an admissible solution.
\end{remark}
By adapting the ideas of Trudinger in \cite{Trudinger1990} for the prescribed curvature equations, we derive the following comparison principle for equation \eqref{app-eq:F-x-D2u} in bounded domains. 

\begin{theorem}\label{thm:comparison-b}
Let $\Omega$ be a bounded domain in $\mathbb{R}^n$. Assume that $u\in\mathrm{USC}(\bar{\Omega})\cap B(\Omega)$ and $v\in\mathrm{LSC}(\bar{\Omega})\cap B(\Omega)$ are respectively viscosity subsolution and viscosity supersolution of equation \eqref{app-eq:F-x-D2u} with conditions \eqref{app-eq:f-increase}--\eqref{app-eq:f-boundary} holding. Then $$\sup_{\Omega}(u-v)=\sup_{\partial\Omega}(u-v).$$
\end{theorem}

To prove Theorem \ref{thm:comparison-b}, we establish below a preliminary result via regularizations.
\begin{lemma}\label{lem:comparison-b}
Let $\Omega$ be a bounded domain in $\mathbb{R}^n$ and let \eqref{app-eq:f-increase}--\eqref{app-eq:f-boundary} hold. Assume that $u\in\mathrm{USC}(\bar{\Omega})\cap B(\Omega)$ and $v\in\mathrm{LSC}(\bar{\Omega})\cap B(\Omega)$ satisfy $$F(x,D^2u)\geq \delta,\quad F(x,D^2v)\leq 0$$
in $\Omega$ in the viscosity sense for some constant $\delta>0$. Then $$\sup_{\Omega}(u-v)=\sup_{\partial\Omega}(u-v).$$
\end{lemma}
\begin{proof}
As in \cite{Jensen-Lions-1988,Urbas1990}, for $\epsilon>0$ we define the approximations of $u$ and $v$ as
\begin{gather}
u_\epsilon^+(x)=\sup_{y\in\Omega}\left\{u(y)-\omega_0\frac{|x-y|^2}{\epsilon^2}\right\},\label{app-eq:approxi-u}\\
v_\epsilon^-(x)=\inf_{y\in\Omega}\left\{v(y)+\omega_0\frac{|x-y|^2}{\epsilon^2}\right\},\label{app-eq:approxi-v}
\end{gather}
where $\omega_0=\max\{\text{osc}_{\Omega}u,\text{osc}_{\Omega}v\}$. The supremum in \eqref{app-eq:approxi-u} and the infimum in \eqref{app-eq:approxi-v} are respectively attained at points $x^{\pm}\in\Omega$ satisfying $|x-x^{\pm}|\leq \epsilon$, provided $x\in\Omega_\epsilon:=\{x\in\Omega:\text{dist}(x,\partial\Omega)>\epsilon\}$. Clearly, $u_\epsilon^+,v_\epsilon^-\in C^{0,1}(\overline{\Omega_\epsilon})$ and 
\begin{equation}\label{app-eq:approxi-u-v}
\sup_{\Omega_\epsilon}|u_\epsilon^+-u|,\ \sup_{\Omega_\epsilon}|v_\epsilon^--v|\to 0\quad\text{as }\epsilon\to0.
\end{equation}
Moreover, the functions $u_\epsilon^+(x)$, $v_\epsilon^-(x)$ are respectively semi-convex and semi-concave in $\Omega$, with
\begin{equation*}
D^2u_\epsilon^+, -D^2v_\epsilon^-\geq-\frac{2\omega_0}{\epsilon^2}
\end{equation*} 
in the sense of distributions, and they satisfy 
\begin{equation*}
F(x^+, D^2u_\epsilon^+)\geq\delta, \quad F(x^-, D^2v_\epsilon^-)\leq 0, 
\end{equation*}
in $\Omega_\epsilon$ in the viscosity sense.

Considering now the semi-convex, Lipschitz continuous function $w_\epsilon=u_\epsilon^+-v_\epsilon^-$, we claim for small $\epsilon$,
\begin{equation}\label{app-eq:w-O-ep}
\sup_{\Omega_\epsilon}w_\epsilon=\sup_{\partial\Omega_\epsilon}w_\epsilon.
\end{equation}
 If not, $w_\epsilon$ has an interior maximum in $\Omega_\epsilon$. By Lemma 3.10 in \cite{Jensen1988}, the upper contact set $K^+$ of $w_\epsilon$ is nonempty in $\Omega_\epsilon$. Thus, for almost all $x\in K^+$, $D^2u_\epsilon^+\leq D^2v_\epsilon^-$. From Remark \ref{app-rk:connection}, we see $\lambda(D^2u_\epsilon^+),\lambda(D^2v_\epsilon^-)\in\Gamma$. Via \eqref{app-eq:f-increase}, we obtain for almost all $x\in K^+$,
\begin{align}
F(x^-,D^2v_\epsilon^-(x))\geq F(x^-,D^2u_\epsilon^+(x))&=F(x^+,D^2u_\epsilon^+(x))+o(1)\notag\\
&\geq\delta+o(1),\label{app-eq:F-ep}
\end{align}
where $o(1)\to 0$ as $\epsilon\to 0$. Given such $x_0\in K^+$, where $v_\epsilon^-$ is twice differentiable, we set for $\tau>0$
\begin{align*}
\phi_\tau(x)=&v_\epsilon^-(x_0)+Dv_\epsilon^-(x_0)(x-x_0)+\frac12(x-x_0)^TD^2v_\epsilon^-(x_0)(x-x_0)\\
&-\frac{\tau}{2}|x-x_0|^2.
\end{align*}
Since $D^2\phi_\tau=D^2v_\epsilon^-(x_0)-\tau I$, $\lambda(D^2\phi_\tau)\in\Gamma$ if $\tau$ is sufficiently small. Also, $v_\epsilon^--\phi_\tau$ has a local minimum at $x_0$. By Definition \ref{app-def:visc},
\begin{equation}\label{app-eq:F-tau}
0\geq F(x_0^-,D^2v_\epsilon^-(x_0)-\tau I)=F(x_0^-,D^2v_\epsilon^-(x_0))+o(1),
\end{equation}
where $o(1)\to0$ as $\tau\to0$. Clearly, \eqref{app-eq:F-tau} contradicts \eqref{app-eq:F-ep} when $\epsilon,\tau$ are small enough.

With facts \eqref{app-eq:approxi-u-v} and \eqref{app-eq:w-O-ep} in hand, we conclude immediately the assertion of Lemma \ref{lem:comparison-b} by letting $\epsilon\to0$.
\end{proof}

\begin{proof}[Proof of Theorem \ref{thm:comparison-b}]
Set $u_t(x)=\frac{1}{t}u(x)$, for $0<t<1$. We first show that $u_t$ satisfies $$F(x,D^2u_t)\geq (1-t)\tilde{\nu}$$
in $\Omega$ in the viscosity sense, where $\tilde{\nu}:=\nu\left(\inf_{\Omega}g\right)$.
Let $\phi$ be an admissible function with a local maximum of $u_t-\phi$ at $x_0$. Since $u$ is a viscosity subsolution of \eqref{app-eq:F-x-D2u}, $$F(x_0,tD^2\phi(x_0))\geq 0.$$
By virtue of \eqref{app-eq:f-nu}, for some $t_0\in(t,1)$,%for an admissible $\phi\in C^2(\Omega)$ and for $t$ sufficienly close to $1$,
\begin{align*}
&F(x_0,D^2\phi(x_0))\\
&=F(x_0,tD^2\phi(x_0))+(1-t)\sum_{i=1}^n\lambda_i(D^2\phi(x_0))\frac{\partial f}{\partial\lambda_i}(t_0\lambda(D^2\phi(x_0)))\\
&\geq (1-t)\nu(f(t\lambda(D^2\phi(x_0))))\geq (1-t)\nu(\inf_{\Omega} g).
\end{align*}
 Applying Lemma \ref{lem:comparison-b}, we have 
\begin{equation*}
\sup_{\Omega}(u_t-v)=\sup_{\partial\Omega}(u_t-v).
\end{equation*}
Consequently, the proof finishes by letting $t\to 1$.
\end{proof}

\begin{remark}
Theorem \ref{thm:comparison-b} may hold for equation \eqref{app-eq:F-x-D2u} without requiring assumption \eqref{app-eq:f-nu}. Indeed, if we let $g\equiv C$ for some constant $C>0$ (or modify assumption \eqref{app-eq:f-increase} as: there exists $\lambda_0>0$ such that 
\begin{equation}\label{app-eq:f-degenerate}
f(\lambda(D^2u+\eta))-f(\lambda(D^2u))\geq \lambda_0(\det\eta)^{\frac{1}{n}}
\end{equation} 
for any admissible function $u$ and symmetric matrix $\eta\geq 0$), then Theorem \ref{thm:comparison-b} could be proved under \eqref{app-eq:f-increase} (\eqref{app-eq:f-degenerate}) and \eqref{app-eq:f-boundary} by applying the Aleksandrov maximum principle. For a view of such arguments,  we refer to \cite{Trudinger1988,Urbas1990}, where uniformly elliptic equations and $k$-Hessian equations ($f=(\sigma_k)^{1/k}$, clearly satisfying \eqref{app-eq:f-degenerate}) were considered.
\end{remark}

Based on Theorem \ref{thm:comparison-b}, one can immediately prove a comparison principle for equation \eqref{app-eq:F-x-D2u} in unbounded domains, provided that the given subsolution and supersolution coincide at infinity. 
\begin{corollary}\label{thm:comparison-ub}
Let $\Omega$ be an unbounded domain in $\mathbb{R}^n$. Assume that $u\in\mathrm{USC}(\bar{\Omega})\cap B_{p}(\Omega)$ and $v\in\mathrm{LSC}(\bar{\Omega})\cap B_{p}(\Omega)$ are respectively viscosity subsolution and viscosity supersolution of equation \eqref{app-eq:F-x-D2u} with conditions \eqref{app-eq:f-increase}--\eqref{app-eq:f-boundary} holding. If $u\leq v$ on $\partial\Omega$ and $$\lim_{|x|\to\infty}(u-v)(x)=0,$$ then $u\leq v$ in $\Omega$.
\end{corollary}
\begin{proof}
For $0<\epsilon\ll 1$, there is $R>0$ such that $u(x)\leq v(x)+\epsilon$ when $|x|>R$. Given arbitrary  $x_0\in\Omega$, choose a large ball $B_r$ with the radius $r>R$ and the center at the origin and  containing $x_0$. Note that $u\leq v+\epsilon$ on $\partial\Omega\cup\partial B_r$. Thus, applying Theorem \ref{thm:comparison-b} to $u$ and $v+\epsilon$ on the domain $\Omega\cap B_r$, we obtain $u(x_0)\leq v(x_0)+\epsilon$. Then letting $\epsilon\to0$ yields $u(x_0)\leq v(x_0)$. The proof is done.
\end{proof}

\section{Perron's method}\label{sec:perron}
With comparison principles established in Theorem \ref{thm:comparison-b} and Corollary \ref{thm:comparison-ub}, Perron's method as in \cite{Ishii1992,Ishii1989,Ishii-Lions-1990} could be immediately adapted to the Dirichlet problem for equation \eqref{app-eq:F-x-D2u} to show the existence of its viscosity solutions. Precisely, we have the following theorem.

\begin{theorem}\label{thm:Perron-m}
Let $\Omega$ be a domain in $\mathbb{R}^n$ and $\varphi\in C^0(\partial\Omega)$. Let $F(x,D^2u)$ be given in \eqref{app-eq:F-x-D2u} with \eqref{app-eq:f-increase}--\eqref{app-eq:f-boundary} holding. Suppose that there exist $\underline{u},\bar{u}\in C^0(\bar{\Omega})$ such that
$$F(x,D^2\underline{u})\geq 0\geq F(x,D^2\bar{u})$$
in $\Omega$ in the viscosity sense, $\underline{u}\leq\bar{u}$ in $\Omega$ and $\underline{u}=\varphi$ on $\partial \Omega$ (and additionally 
$$\lim_{|x|\to\infty}(\underline{u}-\bar{u})(x)=0,$$
provided $\Omega$ is unbounded). Then
\begin{align*}
u(x):=\sup\{&v(x)| v\in\mathrm{USC}(\Omega), F(x,D^2v)\geq 0\text{ in }\Omega\text{ in the viscosity}\\
&\text{sense, with }\underline{u}\leq v\leq\bar{u}\text{ in }\Omega \text{ and }v=\varphi\text{ on }\partial\Omega\}
\end{align*} 
is in $C^0(\bar{\Omega})$ and is a viscosity solution of the problem 
\begin{equation}\label{app-eq:D-F-varphi}
\left\{
\begin{array}{ll}
F(x,D^2u)=0 & \text{in }\Omega,\\
u=\varphi & \text{on }\partial\Omega.
\end{array}
\right.
\end{equation}
\end{theorem}

For a function $v:\Omega\to\mathbb{R}$, we define its upper semicontinuous envelope $v^*$ by $$v^*(x)=\limsup_{y\to x,y\in\Omega}v(y)$$ and lower semicontinuous envelope $v_*$ by $$v_*(x)=\liminf_{y\to x,y\in\Omega}v(y).$$ To prove Theorem \ref{thm:Perron-m}, we need the following key lemmas.

\begin{lemma}[Lemma 4.2 in \cite{Ishii1992}]\label{lem:Perron-sub}
Let $\Omega$ be a domain in $\mathbb{R}^n$ and let $\mathcal{F}$ be a family of viscosity subsolutions of equation \eqref{app-eq:F-x-D2u} in $\Omega$. Let $w(x)=\sup\{u(x):u\in\mathcal{F}\}$ and assume that $w^*(x)<\infty$ for $x\in\Omega$. Then $w^*$ is a viscosity subsolution of \eqref{app-eq:F-x-D2u} in $\Omega$.
\end{lemma}

\begin{lemma}[Lemma 4.4 in \cite{Ishii1992}]\label{lem:Perron-super}
Let $\Omega$ be a domain in $\mathbb{R}^n$ and $u$ be a viscosity subsolution of equation \eqref{app-eq:F-x-D2u} in $\Omega$. If $u_*$ fails to be a viscosity supersolution at some point $\hat{x}$, i.e. there exists admissible function $\psi$ for which $F(\hat{x},D^2\phi(\hat{x}))>0$, then for any small $\kappa>0$ there is a viscosity subsolution $U_{\kappa}$ of \eqref{app-eq:F-x-D2u} satisfying
\begin{equation}\label{app-eq:kappa-super}
\begin{cases}
U_{\kappa}(x)\geq u(x)\quad\text{and}\quad\sup_{\Omega}(U_{\kappa}-u)>0,\\
U_{\kappa}(x)=u(x)\quad\text{for }x\in\Omega, |x-\hat{x}|\geq\kappa.
\end{cases}
\end{equation}
\end{lemma}
\begin{proof}[Proof of Theorem \ref{thm:Perron-m}]
 %The assertion is proved by showing that $u_*$ is a viscosity supersolution of \eqref{app-eq:F-x-D2u} and $u_*=u=u^*$.

First, it is easily observed that $u^*\in\mathrm{USC}(\bar{\Omega})$ and $u_*\in\mathrm{LSC}(\bar{\Omega})$, and that they satisfy $$\underline{u}\leq u_*\leq u\leq u^*\leq\bar{u}\ \text{ in }\Omega\quad\text{and}\quad u_*=u^*=u=\varphi\ \text{ on }\partial\Omega.$$
From Lemma \ref{lem:Perron-sub}, $u^*$ is a viscosity subsolution of \eqref{app-eq:F-x-D2u}. By definition of $u$, $u^*\leq u$, and thus $u=u^*$ in $\bar{\Omega}$, i.e. $u$ is a subsolution of \eqref{app-eq:F-x-D2u}. 

Then let us show $u_*$ is also a viscosity supersolution of \eqref{app-eq:F-x-D2u}. If not, $u_*$ fails to be a supersolution at $\hat{x}\in\Omega$, then in this case we let $U_{\kappa}$ be provided by Lemma \ref{lem:Perron-super}. Clearly, by \eqref{app-eq:kappa-super} $$U_{\kappa}\geq u\geq\underline{u}\ \text{ in }\Omega\quad\text{and}\quad U_{\kappa}=u=\varphi\ \text{ on }\partial\Omega$$ for sufficiently small $\kappa$ (and additionally $$\lim_{|x|\to\infty}(U_\kappa-\bar{u})(x)= \lim_{|x|\to\infty}(u-\bar{u})(x)=0$$ provided $\Omega$ is unbounded). By Theorem \ref{thm:comparison-b} (Corollary \ref{thm:comparison-ub}), $U_\kappa\leq\bar{u}$. Again, by definition of $u$ we deduce $U_\kappa\leq u$ in $\Omega$, contradicting the property $\sup_{\Omega}(U_{\kappa}-u)>0$ in \eqref{app-eq:kappa-super}. Therefore, $u_*$ is a viscosity supersolution. 

Applying now Theorem \ref{thm:comparison-b} (Corollary \ref{thm:comparison-ub}), it follows that $u^*=u\leq u_*$. This yields $u=u_*=u^*$ in $\bar{\Omega}$. Consequently, $u\in C^0(\bar{\Omega})$ with $u=\varphi$ is a viscosity solution of \eqref{app-eq:D-F-varphi}.
\end{proof}


\begin{thebibliography}{1}
\bibitem{Bao2003}
J.G. Bao, J.Y. Chen, B. Guan, M. Ji, Liouville property and regularity of a Hessian quotient equation, Amer. J. Math. 125 (2003), 301-316.
\bibitem{Bao-Li-2013}
J.G. Bao, H.G. Li, The exterior Dirichlet problem for special Lagrangian equations in dimensions $n\leq4$, Nonlinear Anal. 89 (2013), 219-229.

\bibitem{Bao-Li-Li-2014}
J.G. Bao, H.G. Li, Y.Y. Li, On the exterior Dirichlet problem for Hessian equations, Trans. Amer. Math. Soc. 366 (2014), 6183-6200.

\bibitem{Bao-Li-Zhang-2015}
J.G. Bao, H.G. Li, L. Zhang, Monge--Amp\`{e}re equation on exterior domains, Calc. Var. Partial Differential Equations 52 (2015), 39-63. 
 
\bibitem{Caffarelli1995}
L. Caffarelli, Topics in PDEs: The Monge--Amp\`{e}re Equation, Graduate Course, Courant Institute, New York University, 1995. 
\bibitem{Caffarelli-Cabre-1995}
L. Caffarelli, X. Cabr\'{e}, Fully Nonlinear Elliptic Equations, Colloquium Publications,   vol. 43, American Mathematical Society, Providence, RI, 1995.

\bibitem{Caffarelli-Li-2003}
 L. Caffarelli, Y.Y. Li, An extension to a theorem of J\"{o}rgens, Calabi, and Pogorelov, Comm. Pure Appl. Math. 56 (2003), 549-583.
 
\bibitem{Caffarelli1984}
L. Caffarelli, L. Nirenberg, J. Spruck, The Dirichlet problem for nonlinear second-order elliptic equations. I. Monge--Amp\`{e}re equation, Comm. Pure Appl. Math. 37 (1984), 369-402.

\bibitem{Caffarelli1985}
 L. Caffarelli, L. Nirenberg, J. Spruck, The Dirichlet problem for nonlinear second-order elliptic equations. III. Functions of the eigenvalues of the Hessian, Acta Math. 155 (1985), 261-301.
 
\bibitem{Calabi1958}
E. Calabi, Improper affine hyperspheres of convex type and a generalization of a theorem by K. J\"{o}rgens, Michigan Math. J. 5 (1958), 105-126.

\bibitem{Cao-Bao-2017}
X. Cao, J.G. Bao, Hessian equations on exterior domain, J. Math. Anal. Appl. 448 (2017), 22-43.
\bibitem{CY2010}
S.-Y.A. Chang, Y. Yuan, A Liouville problem for sigma-2 equation, Discrete Contin. Dyn. Syst. 28 (2010), 659-664.

\bibitem{Cheng-Yau-1986}
S.Y. Cheng, S.T. Yau, Complete affine hypersurfaces, I. The completeness of affine metrics, Comm. Pure Appl. Math. 39 (1986), 839-866.
\bibitem{Ishii1992}
M.G. Crandall, H. Ishii, P.L. Lions, User's guide to viscosity solutions of second order partial differential equations, Bull. Amer. Math. Soc. 27 (1992), 1-67.


%\bibitem{Deng2019}
%B. Deng, The Neumann problem for a class of fully nonlinear elliptic partial differential equations, arXiv: 1903.04231

\bibitem{Guan1994}
B. Guan, The Dirichlet problem for a class of fully nonlinear elliptic equations,
Comm. Partial Differential Equations 19 (1994), 399-416.

\bibitem{Hardy1952}
G.H. Hardy, J.E. Littlewood, G. P\'olya, Inequalities, second edition, Cambridge University Press, Cambridge, 1952.

%\bibitem{Horn1985}
%R.A. Horn, C.R. Johnson, Matrix analysis, Cambridge University Press, Cambridge, 1985.

\bibitem{Ishii1989}
H. Ishii, On uniqueness and existence of viscosity solutions of fully nonlinear second-order elliptic PDE's, Comm. Pure Appl. Math. 42 (1989), 15-45.

\bibitem{Ishii-Lions-1990}
H. Ishii, P.L. Lions, Viscosity solutions of fully nonlinear second order elliptic partial differential equations, J. Differential Equations 83 (1990), 26-78. 

\bibitem{Ivochkina1983}
N.M. Ivochkina, Classical solvability of the Dirichlet problem for the Monge--Amp\`ere equation, (Russian) Questions in quantum field theory and statistical physics, 4. Zap. Nauchn. Sem. Leningrad. Otdel. Mat. Inst. Steklov. (LOMI) 131 (1983), 72-79.
\bibitem{Ivochkina1985}
N.M. Ivochkina, Solution of the Dirichlet problem for certain equations of Monge--Amp\`ere type, (Russian)
Mat. Sb. (N.S.) 128(170) (1985), 403-415.

\bibitem{Jensen1988}
 R. Jensen, The maximum principle for viscosity solutions of fully nonlinear second order partial differential equations, Arch. Ration. Mech. Anal. 101 (1988), 1-27.
\bibitem{Jensen-Lions-1988}
R. Jensen, P.L. Lions, P.E. Souganidis, A uniqueness result for viscosity solutions of second order fully nonlinear partial differential equations, Proc. Amer. Math. Soc. 102 (1988), 975-978.
\bibitem{Jiang-Li-Li-2020}
T.Y. Jiang, H.G. Li, X.L. Li, On the exterior Dirichlet problem for a class of fully nonlinear elliptic equations, Calc. Var. Partial Differential Equations 60 (2021), Paper No. 17, 20 pp.
\bibitem{Jorgens1954}
K. J\"{o}rgens, \"{U}ber die L\"{o}sungen der Differentialgleichung $rt-s^2=1$ (German), Math. Ann. 127 (1954), 130-134.

\bibitem{Jost2001}
J. Jost, Y. Xin, Some aspects of the global geometry of entire space-like submanifolds, Results Math. 40 (2001), 233-245.
\bibitem{Li-Li-2018}
D.S. Li, Z.S. Li, On the exterior Dirichlet problem for Hessian quotient equations, J. Differential Equations 264 (2018), 6633-6662.

\bibitem{Li-Li-Yuan-2019}
D.S. Li, Z.S. Li, Y. Yuan, A Bernstein problem for special Lagrangian equations in exterior domains, Adv. Math. 361 (2020), 106927, 29 pp.
\bibitem{Li-Bao-2014}
H.G. Li, J.G. Bao, The exterior Dirichlet problem for fully nonlinear elliptic equations related to the eigenvalues of the Hessian, J. Differential Equations 256 (2014), 2480-2501.
\bibitem{Li-Li-Zhao-2019}
H.G. Li, X.L. Li, S.Y. Zhao, Hessian quotient equations on exterior domains, arXiv:2004.06908.   
\bibitem{Li-Ren-Wang-2016}
M. Li, C.Y. Ren, Z.Z. Wang, An interior estimate for convex solutions and a rigidity theorem, J. Funct. Anal. 270 (2016), 2691-2714.
\bibitem{Li-Lu-2018}
Y.Y. Li, S.Y. Lu, Existence and nonexistence to exterior Dirichlet problem for Monge--Amp\`{e}re equation, Calc. Var. Partial Differential Equations  57 (2018), Paper No. 161, 17 pp.
\bibitem{Li2019}
Z.S. Li, On the exterior Dirichlet problem for special Lagrangian equations, Trans. Amer. Math. Soc. 372 (2019), 889-924.
\bibitem{Serrin1960}
N. Meyers, J. Serrin, The exterior Dirichlet problem for second order elliptic partial differential equations, J. Math. Mech. 9 (1960), 513-538.
%\bibitem{Liberman2013}
%G. Lieberman, Oblique boundary value problems for elliptic equations, World Scientific Publishing, 2013. 
%\bibitem{Lions-Trudinger-Urbas-1986}
 %P.L. Lions, N.S. Trudinger, J.I.E. Urbas, The Neumann problem for equations of Monge-Amp\`{e}re type, Comm. Pure Appl. Math. 39 (1986), 539-563. 
%\bibitem{Ma-Qiu-2019}
 %X.N. Ma, G.H. Qiu, The Neumann problem for Hessian equations, Comm. Math. Phys. 366 (2019), 1-28.
\bibitem{Pogorelov1972}
 A.V. Pogorelov, On the improper convex affine hyperspheres, Geom. Dedicata 1 (1972), 33-46.
%\bibitem{Trudinger1987}
%N.S. Trudinger, On degenerate fully nonlinear elliptic equations in balls, Bull. Austral. Math. Soc. 35 (1987), 299-307.
\bibitem{Trudinger1988}
 N.S. Trudinger, Comparison principles and pointwise estimates for viscosity solutions of nonlinear elliptic equations, Rev. Mat. Iberoamericana 4 (1988), 453-468.
\bibitem{Trudinger1990}
N.S. Trudinger, The Dirichlet problem for the prescribed curvature equations, Arch. Ration. Mech. Anal. 111
(1990), 153-179.
\bibitem{Trudinger1995}
N.S. Trudinger, On the Dirichlet problem for Hessian equations, Acta Math. 175 (1995), 151-164.
%\bibitem{Trudinger1997}
%N.S. Trudinger, Weak solutions of Hessian equations, Comm. Partial Differential Equations 22 (1997), 1251-1261.
\bibitem{Urbas1990}
J.I.E. Urbas, On the existence of nonclassical solutions for two class of fully nonlinear elliptic equations, Indiana Univ. Math. J. 39 (1990), 355-382.



\end{thebibliography}
\end{document}